\documentclass[11pt,a4paper,oneside]{amsart}

\usepackage{pdfpages,amsfonts,amsthm,amssymb,latexsym,ucs,graphicx,leftidx}
\usepackage[integrals]{wasysym}

\newtheorem{theorem}{Theorem}

\newtheorem{lemma}{Lemma}

\theoremstyle{definition}

\newcommand{\beql}[1]{\begin{equation}\label{#1}}
\newcommand{\eeq}{\end{equation}}
\newcommand{\comment}[1]{}

\newcommand{\Abs}[1]{{\left|{#1}\right|}}

\newcommand{\Set}[1]{{\left\{{#1}\right\}}}

\newcommand{\RR}{{\mathbb R}}

\newcommand{\CC}{{\mathbb C}}
\newcommand{\ZZ}{{\mathbb Z}}

\newcommand{\inner}[2]{{\langle #1, #2 \rangle}}

\newcommand{\Span}{{\rm span\,}}

\newcommand{\ft}[1]{\widehat{#1}}

\newcounter{rem}
\newcounter{step}
\newcounter{mysec}
\newcounter{mysubsec}[mysec]
\begin{document}

\title{Periodicity of the spectrum of a finite union of intervals}

\author{Mihail N. Kolountzakis}\thanks{Supported by research grant No 3223 from the Univ.\ of Crete}

\date{February 2011}

\address{M.K.: Department of Mathematics, University of Crete, Knossos Ave.,
GR-714 09, Iraklio, Greece} \email{kolount@gmail.com}

\maketitle

\begin{abstract}
A set $\Omega$, of Lebesgue measure 1,
in the real line is called spectral if there is a set $\Lambda$ of real numbers
such that the exponential functions $e_\lambda(x) = \exp(2\pi i \lambda x)$ form
a complete orthonormal system on $L^2(\Omega)$. Such a set $\Lambda$ is called a spectrum of $\Omega$.
In this note we present a simplified proof of the fact
that any spectrum $\Lambda$ of a set $\Omega$ which is finite union of intervals
must be periodic. The original proof is due to Bose and Madan.
\end{abstract}


\noindent{\bf AMS Primary Classification:} 42B99

\section{Introduction and statement of the result}

Let $\Omega \subseteq \RR^d$ be a bounded measurable set of Lebesgue measure 1.
A set $\Lambda \subseteq \RR^d$
is called a spectrum of $\Omega$ (and $\Omega$ is said to be a spectral set) if the set of exponentials
$$
E(\Lambda) = \Set{e_\lambda(x)=e^{2\pi i \lambda\cdot x}:\ \lambda\in\Lambda}
$$
is a complete orthonormal set in $L^2(\Omega)$.
(The inner product in $L^2(\Omega)$ is $\inner{f}{g} = \int_\Omega f \overline{g}$.)

It is easy to see (see, for instance, \cite{kolountzakis2004milano}) that the orthogonality of $E(\Lambda)$
is equivalent to the {\em packing condition}
\beql{packing-condition}
\sum_{\lambda\in\Lambda}\Abs{\ft{\chi_\Omega}}^2(x-\lambda) \le 1,\ \ \mbox{a.e. ($x$)},
\eeq
as well as to the condition
\beql{zeros-condition}
\Lambda-\Lambda \subseteq \Set{0} \cup \Set{\ft{\chi_\Omega}=0}.
\eeq
The completeness of $E(\Lambda)$ is in turn equivalent to the {\em tiling condition}
\beql{tiling-condition}
\sum_{\lambda\in\Lambda}\Abs{\ft{\chi_\Omega}}^2(x-\lambda) = 1,\ \ \mbox{a.e. ($x$)}.
\eeq
These equivalent conditions follow from the identity
\beql{inner}
\inner{e_\lambda}{e_\mu} = \int_\Omega e_\lambda \overline{e_\mu} = \ft{\chi_\Omega}(\lambda-\mu)
\eeq
and from the completeness of all the expontials in $L^2(\Omega)$.

{\em Example:} If $Q_d = (-1/2, 1/2)^d$ is the cube of
unit volume in $\RR^d$ then
$\ZZ^d$ is a spectrum of $Q_d$.

In the one dimensional case, which will concern us in this paper,
condition \eqref{zeros-condition}
implies that the set $\Lambda$ has gaps bounded below by a positive number, the smallest
zero of $\ft{\chi_\Omega}$.

Research on spectral sets has been driven for many years by a conjecture of Fuglede
\cite{fuglede1974operators} which stated that a set $\Omega$ is spectral if and only
if it is a translational tile. A set $\Omega$ is a translational tile if
we can translate copies of $\Omega$ around and fill space without overlaps.
More precisely there exists a set $S \subseteq \RR^d$ such that
\beql{tiling}
\sum_{s\in S} \chi_\Omega(x-s) = 1,\ \ \mbox{a.e. ($x$)}.
\eeq
This conjecture is now known to be false in both directions if $d\ge 3$
\cite{tao2004fuglede,matolcsi2005fuglede4dim,kolountzakis2006hadamard,kolountzakis2006tiles,farkas2006onfuglede,farkas2006tiles}
and both directions are still open in dimensions $d=1,2$.

In this paper we present a new proof of the periodicity of the spectrum,
which is a considerable simplification of that in \cite{bose2010spectrum}.
\begin{theorem}[Bose and Madan \cite{bose2010spectrum}]\label{th:main}
If $\Omega = \bigcup_{j=1}^n (a_j, b_j) \subseteq \RR$ is a finite union of intervals of total
length $1$ and $\Lambda \subseteq \RR$ is a spectrum of $\Omega$ then there exists a positive integer
$T$ such that $\Lambda+T=\Lambda$.
\end{theorem}
This is the spectral analogue of a result \cite{lagarias1996tiling,kolountzakis1996structure}
which states that all translational tilings by a bounded measurable set (or by a compactly supported
function) are necessarily periodic. The proof of Theorem \ref{th:main} is given in the next section.

\section{Proof of the periodicity of the spectrum}

Let us observe first, as in \cite{bose2010spectrum}, that the spectrum
$\Lambda = \Set{\lambda_j:\ j\in\ZZ}$, $\lambda_j<\lambda_{j+1}$,
of any bounded set $\Omega \subseteq \RR$
has ``finite complexity'', in the sense that all gaps $\lambda_{j+1}-\lambda_j$ are
drawn from the discrete set ($\ft{\chi_\Omega}$ is analytic as $\chi_\Omega$
has bounded support) $\Set{\ft{\chi_\Omega}=0}$.
This implies that if we consider all
intesections of $\Lambda$ with a sliding window of width $h$
$$
[\lambda, \lambda+h] \cap \Lambda,\ \ \mbox{(where $\lambda\in\Lambda$)},
$$
then we only see finitely many different sets.

If $\Omega = \bigcup_{j=1}^n (a_j, b_j) \subseteq \RR$ it follows by a simple calculation that
\beql{ft}
\ft{\chi_\Omega}(\xi) = \frac{1}{2\pi i \xi}\sum_{j=1}^n \left( e^{-2\pi i a_j \xi} - e^{-2\pi i b_j \xi} \right).
\eeq

The important ingredient of the approach in \cite{bose2010spectrum} that we keep in our approach
is the view of the spectrum as a linear space via the map
$\phi = \phi_\Omega:\RR\to\CC^{2n}$ given by
$$
x \to (e^{-2\pi i a_1 x}, \ldots, e^{-2\pi i a_n x}, e^{-2\pi i b_1 x}, \ldots, e^{-2\pi i b_n x}).
$$
Define the bilinear form $A$ on $\CC^{2n}$ by (writing $z=(z_1, z_2)$, $z_1, z_2 \in \CC^n$)
$$
A(z,w) = \inner{z_1}{w_1}-\inner{z_2}{w_2},
$$
where $\inner{\cdot}{\cdot}$ is the usual inner product on $\CC^n$. 
Using \eqref{ft} we see that if $\lambda\neq\mu$ then
$$
e_\lambda \perp e_\mu\ \ \mbox{if and only if}\ \ A(\phi(\lambda), \phi(\mu))=0.
$$
Write
$$
V(\Lambda) = \Span \phi(\Lambda)
$$
for the subspace of $\CC^{2n}$ generated by
the set $\phi(\Lambda) = \Set{\phi(\lambda):\ \lambda\in\Lambda}$.

Suppose now that $B=\Set{b_1,\ldots,b_m} \subseteq \Lambda$ is a generating set, i.e., that
$V(\Lambda) = \Span\phi(B)$. It follows that $x \in \Lambda$ if and only if $A(\phi(x), \phi(b_j))$
for $j=1,2,\ldots,m$. Indeed, if the latter condition is true it follows by linearity that
$A(\phi(x),\phi(\mu))=0$ for all $\mu \in \Lambda$
and hence that $e_x \perp e_\mu$, $\mu\in\Lambda\setminus\Set{x}$.
This implies that $x\in\Lambda$, otherwise $E(\Lambda)$ would not be a complete set
of exponentials for $L^2(\Omega)$.
As remarked in \cite{bose2010spectrum} this means that $\Lambda$ is determined by any such generating
set $B$.

\begin{lemma}\label{lm:tails}
Let $\Omega$ be a finite union of intervals.
If $A\subseteq\RR$ is a set of positive minimum gap
$\delta$ then for $R>0$ we have
$$
\sum_{a \in A \atop \Abs{a} > R} \Abs{\ft{\chi_\Omega}}^2(a) \le C/R,
$$
for some constant $C>0$ that may depend on $\Omega$ and $\delta$ only.
\end{lemma}
\begin{proof}
This is immediate from the fact that $\Abs{\ft{\chi_\Omega}}^2(y) \le C/\Abs{y}^2$ (see \eqref{ft}).
\end{proof}

\begin{lemma}\label{lm:gen-set}
There is a finite $T>0$ such that for all $x\in\RR$ the set $\Lambda\cap(x, x+T)$ is a generating set.
\end{lemma}
\begin{proof}
Suppose not, so that there is a sequence $m_k \in \Lambda$, $k=1,2,\ldots$, such that
$$
\dim \Span \phi(\Lambda\cap(m_k-k, m_k+k)) < \dim \Span \phi(\Lambda).
$$
Consider the sequence of finite sets
$$
M_k = \left[ \Lambda\cap(m_k-k, m_k+k) \right] - m_k,
$$
i.e., the sets $\Lambda\cap(m_k-k, m_k+k)$ translated so that they are centered at $0$ (therefore
they all contain $0$).
Observe that in any given interval $(-t, t)$ the sets $M_k$ may only take finitely many forms.

For $n=1,2,3,\ldots$ in turn
we look at the infinite sequence
$$
M_k \cap (-n,n),\ \ \ k=1,2,\ldots.
$$
There is an infinite sequence of $k$'s such
that all sets $M_k \cap (-n, n)$ are the same. Keep only these indices and define $L_n$
to be this common set.
In this way we define an increasing infinite sequence of sets $L_n$, $L_n \subseteq L_{n+1}$,
each of which contains $0$ and is of the form
$$
L_n = \Lambda\cap(c_n-n, c_n+n) - c_n,
$$
for some $c_n\in\Lambda$.

Let $L=\bigcup_{n=1}^\infty L_n$.
Since each finite part of $L$ is a translate of a part of $\Lambda$
it follows that the elements of $E(L)$ are orthogonal.
We now show that $E(L)$ is also complete and is thus also a spectrum of $\Omega$.

For this it suffices to show that $F(x) := \sum_{\ell \in L} \Abs{\ft{\chi_\Omega}}^2(x-\ell)=1$
for almost every $x\in\RR$. Assume for simplicity that $x\ge 0$. We have for $t>2x$
\begin{align*}
1 &\ge F(x) & \mbox{(from \eqref{packing-condition}, since $E(L)$ is an orthogonal set)}\\
 &\ge \sum_{\ell \in (-t,t)\cap L} \Abs{\ft{\chi_\Omega}}^2(x-\ell)\\
 &= \sum_{\ell \in (-t,t)\cap L_n} \Abs{\ft{\chi_\Omega}}^2(x-\ell) & \mbox{(for some $n=n(t)>t$)}\\
 &= \sum_{\ell \in \Lambda-c_n,\ \Abs{\ell}<t} \Abs{\ft{\chi_\Omega}}^2(x-\ell)\\
 &= 1 -\sum_{\ell \in \Lambda-c_n,\ \Abs{\ell}\ge t} \Abs{\ft{\chi_\Omega}}^2(x-\ell) & \mbox{(by \eqref{tiling-condition}, since $\Lambda$ is a spectrum)}\\
 &\ge 1 - \sum_{\ell \in \Lambda-c_n,\ \Abs{x-\ell}\ge t/2} \Abs{\ft{\chi_\Omega}}^2(x-\ell) & \mbox{(as $\Abs{\ell}\ge t > 2x$ implies $\Abs{x-\ell}\ge t/2$)}\\
 &= 1 - \sum_{a \in x-\Lambda+c_n,\ \Abs{a}\ge t/2} \Abs{\ft{\chi_\Omega}}^2(a) & \mbox{(with $a=x-\ell$)}\\
 &\ge 1-\frac{C}{t} & \mbox{(from Lemma \ref{lm:tails} applied to the set $x-\Lambda+c_n$)}.
\end{align*}
Letting $t\to\infty$ we obtain that $F(x)=1$ for all $x\in\RR$.
(Notice that the constant $C$ that appears above does
not depend on $n$.)

Since every finite subset of $L$ is contained in some $L_n$ it follows
that
\beql{smaller-dim}
\dim\Span\phi(L) < \dim\Span\phi(\Lambda).
\eeq

To derive a contradiction let the finite set
$\Lambda' \subseteq \Lambda$ be such that $\phi(\Lambda')$ is a basis of $\Span \phi(\Lambda)$ and
also let the finite set $L' \subseteq L$ be such that $\phi(L')$ is a basis of $\Span \phi(L)$.
Some translate $s+L'$ of the finite set $L'$ is contained in $\Lambda$, hence 
$$
A(\phi(s+\ell'), \phi(\lambda')) = 0,\ \ \ \mbox{(for all $\ell' \in L'$ and $\lambda' \in \Lambda'$)},
$$
which implies
$$
A(\phi(\ell'), \phi(-s+\lambda')) = 0,\ \ \ \mbox{(for all $\ell' \in L'$ and $\lambda' \in \Lambda'$)},
$$
and this means that $-s+\Lambda' \subseteq L$ and therefore that
$$
\dim \Span \phi(L) \ge \dim \Span \phi(-s+\Lambda') = \dim \Span \phi(\Lambda') = \dim \Span \phi(\Lambda),
$$
in contradiction with \eqref{smaller-dim}.
We have used the easy fact that $\dim\Span\phi(A+x)=\dim\Span\phi(A)$ for any $x\in\RR$, $A\subseteq\RR$.
\end{proof}

\noindent
{\bf Completion of the proof: }The set $\Lambda$ is periodic.

Let $T$ be as in Lemma \ref{lm:gen-set} and consider all subsets of $\Lambda$ of the form
$$
B_\lambda = \Lambda \cap [\lambda, \lambda+T],\ \ \ \lambda \in \Lambda.
$$
It follows from Lemma \ref{lm:gen-set} that $B_\lambda$ is a generating set for each $\lambda$.
But there are only finitely many different forms the set $B_\lambda-\lambda$ can take, hence there are
$\lambda_1, \lambda_2 \in \Lambda$, $\lambda_1 > \lambda_2$, such that
$$
B_{\lambda_1}-\lambda_1 = B_{\lambda_2}-\lambda_2,
$$
or
$$
B_{\lambda_1} = B_{\lambda_2} + \lambda_1-\lambda_2.
$$
Since $B_{\lambda_1}$ and $B_{\lambda_2}$ are both generating sets for $\phi(\Lambda)$
it follows that
\begin{align*}
x \in \Lambda &\Leftrightarrow e_x \perp e_y\ \ (y\in B_{\lambda_2})\\
 &\Leftrightarrow e_{x+\lambda_1-\lambda_2} \perp e_y \ \ (y\in B_{\lambda_1})\\
 &\Leftrightarrow x+(\lambda_1-\lambda_2) \in \Lambda.
\end{align*}
In other words, $T=\lambda_1-\lambda_2$ is a period of $\Lambda$.

Let us also remark that any period of $\Lambda$ must be an integer. This is a consequence of
the fact that $\Lambda$ has density 1: if $T$ is a period of $\Lambda$ this implies that
there are exactly $T$ elements of $\Lambda$ in each interval $[x, x+T)$ hence $T$
is an integer.

\bibliographystyle{abbrv}
\bibliography{spectral-sets}

\begin{thebibliography}{10}

\bibitem{bose2010spectrum}
D.~Bose and S.~Madan.
\newblock {Spectrum is periodic for $n$-Intervals}.
\newblock {\em Journal of Functional Analysis}, 260(1):308--325, 2011.

\bibitem{farkas2006onfuglede}
B.~Farkas, M.~Matolcsi, and P.~M\'ora.
\newblock On {Fuglede}'s conjecture and the existence of universal spectra.
\newblock {\em J. Fourier Anal. Appl.}, 12(5):483--494, 2006.

\bibitem{farkas2006tiles}
B.~Farkas and S.~R{\'e}v{\'e}sz.
\newblock Tiles with no spectra in dimension 4.
\newblock {\em Math. Scand.}, 98(1):44--52, 2006.

\bibitem{fuglede1974operators}
B.~Fuglede.
\newblock Commuting self-adjoint partial differential operators and a group
  theoretic problem.
\newblock {\em J. Funct.\ Anal.}, 16:101--121, 1974.

\bibitem{kolountzakis2004milano}
M.~Kolountzakis.
\newblock {The study of translational tiling with Fourier Analysis}.
\newblock In L.~Brandolini, editor, {\em Fourier Analysis and Convexity}, pages
  131--187. Birkh\"auser, 2004.

\bibitem{kolountzakis1996structure}
M.~Kolountzakis and J.~Lagarias.
\newblock {Structure of tilings of the line by a function}.
\newblock {\em Duke Mathematical Journal}, 82(3):653--678, 1996.

\bibitem{kolountzakis2006hadamard}
M.~Kolountzakis and M.~Matolcsi.
\newblock Complex {Hadamard} matrices and the spectral set conjecture.
\newblock {\em Collect.\ Math.}, Extra:281--291, 2006.

\bibitem{kolountzakis2006tiles}
M.~Kolountzakis and M.~Matolcsi.
\newblock {Tiles with no spectra}.
\newblock {\em Forum Math.}, 18:519--528, 2006.

\bibitem{lagarias1996tiling}
J.~Lagarias and Y.~Wang.
\newblock {Tiling the line with translates of one tile}.
\newblock {\em Inventiones Mathematicae}, 124(1):341--365, 1996.

\bibitem{matolcsi2005fuglede4dim}
M.~Matolcsi.
\newblock Fuglede's conjecture fails in dimension 4.
\newblock {\em Proc. Amer. Math. Soc.}, 133(10):3021--3026, 2005.

\bibitem{tao2004fuglede}
T.~Tao.
\newblock {Fuglede's conjecture is false in 5 and higher dimensions}.
\newblock {\em Math. Res. Lett.}, 11(2-3):251--258, 2004.

\end{thebibliography}

\end{document}